\documentclass[reqno]{amsart}
\usepackage{amsmath,amsrefs,color}

\usepackage{geometry} 
\usepackage{amssymb}
\usepackage[bottom]{footmisc}

\numberwithin{equation}{section}

\renewcommand{\(}{\left(}
\renewcommand{\)}{\right)}

\newtheorem{theorem}{Theorem}[section] 
\newtheorem{lemma}[theorem]{Lemma}     
\newtheorem{corollary}[theorem]{Corollary}

\newtheorem{remark}[theorem]{Remark}

\begin{document}
	
	\title[Gradient estimates for an elliptic equation]{Gradient estimates for nonlinear elliptic equations with a gradient-dependent nonlinearity}
	
		\author{Joshua Ching and Florica C. C\^irstea}

	\address{School of Mathematics and Statistics, The University of Sydney, NSW 2006, Australia} 
	\email{joshua.ching@sydney.edu.au}
	\email{florica.cirstea@sydney.edu.au}

\let\thefootnote\relax\footnote{F.C. C\^{\i}rstea was partially 
		supported by ARC Discovery grant number DP120102878 ``Analysis of non-linear partial differential equations
		describing singular phenomena".}

\begin{abstract} In this paper, we obtain gradient estimates of the positive solutions to weighted $p$-Laplacian type equations with a 
	gradient-dependent nonlinearity of the form \begin{equation} \label{unu}
	{\rm div} \(|x|^{\sigma}|\nabla u|^{p-2} \nabla u\)= |x|^{-\tau} u^q |\nabla u|^m \quad \mathrm{in } \ \Omega^*:= \Omega \setminus \{ 0 \}. 
	\end{equation}
	Here, $\Omega\subseteq \mathbb R^N$ denotes a domain containing the origin with $N\geq 2$, whereas $m,q\in [0,\infty)$, $1<p\leq N+\sigma$ and 
	$q>\max\{p-m-1,\sigma+\tau-1\}$. The main difficulty arises from the dependence of the right-hand side of \eqref{unu} on  $x$, $u$ and $|\nabla u|$, without any upper bound restriction on the power $m$ of $|\nabla u|$. Our proof of the gradient estimates is based on a two-step process relying on a modified version of the Bernstein's method. As a by-product, we extend the range of applicability of the Liouville-type results known for \eqref{unu}.  
\end{abstract}

\maketitle
\section{Introduction and main result}

{\em A priori} estimates for second-order, nonlinear elliptic and parabolic equations are of fundamental importance in geometry and partial differential equations. Independent of any knowledge of the existence of solutions, {\em a priori} estimates   
play a crucial role in establishing existence, uniqueness, regularity and other qualitative properties of solutions. For example, a key step for proving the 
existence of solutions for quasilinear elliptic equations is represented by local or global gradient bounds. 
Such {\em a priori} estimates lead to Harnack inequalities, Liouville theorems and compactness theorems for linear and nonlinear partial differential equations. 

In this paper we derive gradient bounds for the positive solutions to a class of elliptic equations in divergence form such as
\begin{equation} \label{le}
{\rm div} (|x|^{\sigma}|\nabla u|^{p-2} \nabla u)= |x|^{-\tau} u^q |\nabla u|^m \quad \mathrm{in } \ \Omega^*:= \Omega \setminus \{ 0 \}.  
\end{equation}
Here, $\Omega\subseteq  \mathbb{R}^N$ denotes a domain containing the origin with $N\geq 2$, while $m,p,q,\sigma$ and $ \tau$ are real parameters.  
We define $k$ and $\ell$ by
\begin{equation} \label{kl}
k:=m+q-p+1\ \mathrm{and}\  \ell:=q+1-\sigma-\tau. 
\end{equation}
We assume throughout the following condition 
\begin{equation} \label{GG}
1<p\leq  N+\sigma,  \quad \min\{k,\ell\}>0 \quad \text{and}\quad 
m,q\in [0,\infty). 
\end{equation}

By a positive solution $u$ of (\ref{le}) we mean a positive function $u\in C^2(\Omega^*)$ satisfying (\ref{le}) in the classical sense. 
By the strong maximum principle (see Lemma~\ref{smp11}), any non-negative and non-zero solution of (\ref{le}) is positive in $\Omega^*$.

The main result of this paper is the following theorem.

\begin{theorem} \label{thm1} 
	Let {\rm(\ref{GG})} hold. 
	There exists a positive constant $C=C(m,N,p,q,\sigma, \tau)$ such that for any positive solution $u$ of \eqref{le} and any $r_0>0$ with 
	$\overline{B_{2 r_0}(0)} \subset \Omega$, it holds
	\begin{equation} \label{gu2a}
	|\nabla u (x)| \leq C |x|^{-\frac{\ell}{k}} \quad \mathrm{for \ every }\ 0<|x| \leq r_0. 
	\end{equation}
	In particular, if $\Omega = \mathbb{R}^N$, then {\rm(\ref{gu2a})} holds for all $x \in \mathbb{R}^N \setminus \{ 0 \}$.
\end{theorem}

Theorem~\ref{thm1} is also applicable if $m=0$. In this instance and other  
particular cases of \eqref{GG}, by assuming an upper bound on $m$, gradient estimates can be obtained by deriving {\em a priori} estimates of the solutions, then using a suitable transformation and \cite{Tolksdorf}*{Theorem 1}, see \cite{FriedmanVeron}*{Lemma 1.1} for
$m=\sigma=\tau=0$ and \cite{CC}*{Lemma 3.8} for $\sigma=\tau=0<m<2=p$. 
We generalise such estimates to (\ref{le}) in Theorem~\ref{thm1}
through a different method (akin to that in \cites{Lions,Nguyen}) without any upper bound restriction on the power $m$ of $|\nabla u|$ in the right-hand side of (\ref{le}).

An important tool for obtaining gradient bounds is 
the classical Bernstein's method, introduced by Bernstein (\cite{Be1}--\cite{Be3}) at the beginning of the 20th century. The basic idea is to derive a differential equation for $|\nabla u|^2$ and then apply the maximum principle. 
Bernstein's method was substantially developed by Ladyzhenskaya \cite{LA} and 
Ladyzhenskaya and Ural'tseva \cites{LU1,LU2} 
(to obtain both interior and global gradient estimates for uniformly elliptic equations)
and later applied systematically to quasilinear elliptic equations by Serrin \cite{Se}, Lions \cite{Lions} and many others, leading to a definitive quasilinear theory as described by Gilbarg and Trudinger \cite{GT}. A weak Bernstein method was introduced by Barles \cite{Ba} for fully nonlinear elliptic equations based on the approach of viscosity solutions. 

We now return to equation (\ref{le}) for a brief review of gradient estimates. 
{\em Without the factor $u^q$} in (\ref{le}), by relying on the Bernstein technique, 
{\em a priori} gradient bounds were first derived by Lions \cite{Lions}*{Theorem IV.1} for $\Delta u=|\nabla u|^m $  and 
recently extended by Nguyen \cite{Nguyen}*{Lemma 2.2} to include equations such as $\Delta u=|x|^{-\tau}|\nabla u|^m$ in $\Omega^*$ when $m>1>\tau$. {\em A priori} universal gradient estimates for the quasilinear elliptic equation $ {\rm div} (|\nabla u|^{p-2} \nabla u)= |\nabla u|^m$ 
with $m>p-1>0$  on a domain $\Omega$ of $\mathbb R^N$ ($N\geq p$) have been obtained by Bidaut-V\'eron et al. \cite{BVGH}*{Proposition~2.1}. They extended their estimates to equations on complete non-compact manifolds satisfying a lower bound estimate on the Ricci curvature and used them to derive Liouville type theorems. 

We aim to generalise the gradient estimates in \cite{Nguyen}*{Lemma~2.2} and \cite{BVGH} to the {\em weighted} $p$-Laplacian type equation 
\eqref{le} in the corresponding framework of \eqref{GG}. New difficulties arise due to the introduction of a non-negative power $u^q$ and of a weight function $|x|^{-\tau}$ in the right-hand side of (\ref{le}). We next outline the main steps in the derivation of the gradient bounds of Theorem~\ref{thm1} for any positive solution $u_1$ of \eqref{le}. 
Fix $x_0\in B_{r_0}(0)\setminus\{0\} $ such that 
$|\nabla u_1(x_0)|>0$. 
Let $\mathcal G$ denote the maximal connected component of $\{x\in \Omega^*:\  |\nabla u_1(x)|>0\}$ containing $x_0$.
We set $\rho_0:=|x_0|$ and $a_{1,1}(x):=|x|^{-\tau -\sigma}  u_1^q(x) |\nabla u_1(x)|^{m+2-p}$ for $x\in \mathcal G\cap B_{\rho_0/2}(x_0)$ so that 
\begin{equation} \label{gs} \Delta u_1=a_{1,1}(x) -\sigma\frac{\langle x,\nabla u_1\rangle}{|x|^2}-\left(p-2\right)  \frac{\langle 
	(D^2 u_1)(\nabla u_1),\nabla u_1\rangle}{|\nabla u_1|^2}\quad \text{on }\  \mathcal G\cap B_{\rho_0/2}(x_0).
\end{equation} Note that the power factor $u_1^q$ is hidden into $a_{1,1}$. 
Let $\phi$ and $\omega$ be given by \eqref{pf1} and \eqref{omegadef}, respectively. 
By ``linearising" \eqref{gs}, we need to prescribe a suitable linear operator $\mathcal L_1 [w]$ for $w \in C^2(\mathcal G \cap B_{\rho_0/2}(x_0))$ and be able to bound $\mathcal L_1 [w_1]$ from above on $\omega$ for $$w_1 = \phi^{2 \alpha_1}|\nabla u_1|^2,$$ where $\alpha_1$ is a positive constant to be conveniently chosen as $1/(2k)$. By the definition of $\phi$ and $\omega$, we have that $\max_{\overline \omega} w_1=w_1(x^*)$ for some 
$x^*\in \omega$. Since $\nabla w_1(x^*)=0$ and $(D^2 w_1)(x^*)$ is negative semi-definite, the definition of $\mathcal L_1 [w_1]$ at $x^*$ will lead to 
$\mathcal L_1 [w_1](x^*)\geq 0$ (see (\ref{max11})). 

The construction of $\mathcal L_1$ is a critical step, which becomes significantly more difficult than for the Laplacian type equations treated in \cites{Lions,Nguyen}. For our more intricate weighted $p$-Laplacian type equation (\ref{le}), the nonlinearity depends on the unknown $u$ and its gradient, and also on the space variable $x$. By denoting $z_1(x)=|\nabla u_1(x)|^2$ for every $x\in \mathcal G\cap B_{\rho_0/2}(x_0)$ and 
$$ \mathcal A_1[w]:= -\Delta w- \sigma\frac{\langle x,\nabla w\rangle}{|x|^2}-\left(p-2\right)  \frac{\langle 
	(D^2 w)(\nabla u_1),\nabla u_1\rangle}{|\nabla u_1|^2}, 
$$
our operator $\mathcal L_1[w]$ is defined by 
\begin{equation} \label{l1def} \mathcal L_1[w]:=\mathcal A_1[w]+\left(m+2-p\right) a_{1,1}\frac{\langle \nabla u_1,\nabla w\rangle}{z_1}-M\frac{\langle \nabla z_1,\nabla w\rangle}{z_1}.\end{equation}
Here, $M$ is a sufficiently large constant (as in Lemma~\ref{lm11}). 
Compared with  \cites{Lions,Nguyen} (where $p=2$ and $\sigma=q=0$), our operator $\mathcal L_1$ in (\ref{l1def}) introduces the extra term 
$-M\frac{\langle \nabla z_1,\nabla w\rangle}{z_1}$, a trick inspired by the work of Bidaut-V\'eron et al. \cite{BVGH}*{Proposition~2.1}. 
We mention that Bernstein's method is adapted differently in \cite{BVGH}  than in this paper.  
In Lemma~\ref{lm11}, we  
bound $\mathcal L_1[w_1]$ from above to get positive constants $d_i(m,N,p,q,\sigma,\tau)$ for $i=1,2,3$ such that
\begin{equation} \label{haha} \mathcal L_1[w_1]\leq \left[d_1\frac{z_1}{\phi |x|^2}-d_2 \left(a_{1,1}(x)\right)^2-d_3\frac{|\nabla z_1|^2}{z_1}\right]\phi^{2\alpha_1}\leq d_1 \phi^{\frac{1}{k}-1}\frac{z_1}{|x|^2}-d_2 u_1^{2q} \phi^{\frac{1}{k}}\frac{z_1^{m+2-p}}{|x|^{2(\tau+\sigma)}}\end{equation} 
for all $ x\in \omega$ since 
$\alpha_1=1/(2k)$. The right-hand side of \eqref{haha} shows that we cannot proceed further without an intermediate estimate that relates the gradient of the solution to the solution itself. This is done in Lemma~\ref{lm1} by an appropriate log transform of $u_1$ in \eqref{defu2} to obtain a new function $u_2$ satisfying (\ref{vee}). Up to a constant, this transformation combines the powers of $u_1$ with powers of $|\nabla u_1|$  (see the definition of $h_2$ in \eqref{hj}) into the exponential term $e^{-k u_2}$.  As a result, we can similarly modify the Bernstein method as for $u_1$ to obtain a gradient estimate for $u_2$ using $\mathcal L_2$ in \eqref{L1} applied to $w_2$ given by \eqref{wj12}. Hence, we can derive the intermediate estimate \eqref{gu1a} in Lemma~\ref{lm1}, 
which employed in \eqref{haha} leads to \eqref{j23}. Since $\mathcal L_1 [w_1](x^*)\geq 0$, by letting $x=x^*$ in \eqref{j23}, we finally reach 
the main estimate in \eqref{gu2a} for the solution $u_1$ of \eqref{le}.

\vspace{0.2cm}
{\bf Structure of the paper.} In Section~\ref{sec2}, we provide the main ingredients in the proof of Theorem~\ref{thm1}.  We postpone the technicalities of the proof to Section~\ref{app1}. 
By applying Theorem~\ref{thm1}, we obtain (i) a Liouville-type result in Corollary~\ref{guec1}, which improves through a different method the corresponding results in Farina and Serrin \cite{FarSer}*{Theorems 2,3} and (ii) {\em a priori} estimates of the positive solutions of \eqref{le} in Corollary~\ref{ch4ap1}.

\section{Proof of the main result} \label{sec2}
 
As explained in the introduction, a crucial step in establishing \eqref{gu2a} is the estimate in \eqref{gu1a} relating the gradient of an arbitrary positive solution $u_1$ of \eqref{le} to the solution $u_1$ itself. 

\vspace{0.2cm}
\begin{lemma} \label{lm1}
	Let  \eqref{GG} hold. There exists a positive constant $C_1=C_1(m,N,p,q,\sigma,\tau)$ such that for every positive solution $u_1$ of \eqref{le} and any $r_0>0$ with $\overline{B_{2 r_0}(0)} \subset \Omega$, we have that
	\begin{equation} \label{gu1a}
	|\nabla u_1 (x)| \leq C_1 \frac{u_1(x)}{|x|} \quad \mathrm{for \ all}\  0<|x|\leq r_0.
	\end{equation}
\end{lemma}

\begin{proof} Fix $x_0\in \mathbb{R}^N$ with $0<|x_0| \leq r_0$. To prove \eqref{gu2a} and \eqref{gu1a} at $x=x_0$ for a positive solution $u_1$ of \eqref{le}, we assume that 
	$|\nabla u_1(x_0)|>0$. 
	Let $\mathcal G$ denote the maximal connected component of the set $\{x\in \Omega\setminus\{0\}:\  |\nabla u_1(x)|>0\}$ containing $x_0$.
	We set $\rho_0:=|x_0|$.  
	Let $C_0>0$ be a small constant 
	such that $C_0 u_1(x) < 1$ for every $\rho_0/2  \leq |x| \leq 3 \rho_0 /2$. We define 
	\begin{equation} \label{defu2} u_2(x):=-\log (C_0 u_1(x))\quad \mathrm{for \ every}\  
	|x|\in (\rho_0/2 , 3 \rho_0 /2). 
	\end{equation}
	We set $z_j(x)=|\nabla u_j(x)|^2>0$ for $j=1,2$ and $x\in \mathcal G\cap B_{\rho_0/2}(x_0)$. 
	For any $t>0$, we denote
	\begin{equation} \label{hj}
	f(x,t):=|x|^{-\tau-\sigma} t^{\frac{m+2-p}{2}} \ \mathrm{for } \ x\in \mathbb R^N\setminus\{0\}\ \mathrm{and}\ 
	h_j(t):=\left\{ \begin{aligned}
	& t^q &&\mathrm{if}\ j=1,&\\
	& -C_0^{-k} e^{-kt} && \mathrm{if}\ j=2.& 
	\end{aligned} 
	\right.
	\end{equation}
	For $j=1,2$ and $ x \in \mathcal G\cap B_{\rho_0/2}(x_0)$, let $a_{0,j,w}(x)$, $a_{1,j,w}(x)$ and $\mathcal A_j [w]$ be given by
	\begin{equation} \label{op}
	\left\{ \begin{aligned}
	& a_{0,j,w}(x):=(j-1) (p-1)\langle \nabla u_j,\nabla w\rangle,
	\quad a_{1,j,w}(x):=h_j(u_j)  f(x,z_j) \frac{\langle \nabla u_j,\nabla w\rangle}{z_j},\\
	& \mathcal A_j [w](x):=-\Delta w(x) +a_{2,w}(x)+a_{3,j,w}(x),\ \text{where we define}\\
	& a_{2,w}(x):=- \sigma \frac{\langle x,\nabla w \rangle}{|x|^2},\quad
	a_{3,j,w}(x):=-(p-2) \frac{\langle (D^2 w)(\nabla u_j), \nabla u_j \rangle}{z_j}.
	\end{aligned} \right.
	\end{equation} 
	When $w=u_j$ in $a_{i,j,u_j}$ for $i=0,1,3$, we simply write $a_{i,j}$. For symmetry of notation, we also use $a_{2,j}$ instead of 
	$a_{2,u_j}$. In particular, since $\nabla z_j=2(D^2 u_j)(\nabla u_j)$, we have \begin{equation} \label{a01} 
	a_{0,j}=(j-1)(p-1)z_j,\quad a_{1,j}(x)=h_j(u_j)  f(x,z_j),\quad a_{3,j}=-\frac{(p-2)}{2}\frac{\langle \nabla z_j,\nabla u_j\rangle}{z_j} .\end{equation} 
	Then, $u_j$ (with $j=1,2$) satisfies the equation
	\begin{equation} \label{vee}
	\Delta u_j =
	\sum_{i=0}^3 a_{i,j}(x) \quad \text{for all } x \in \mathcal G\cap B_{\rho_0/2}(x_0).
	\end{equation}
	Next, for $j=1,2$ and $ x \in \mathcal G\cap B_{\rho_0/2}(x_0)$, we introduce the operator for $w \in C^2 (\mathcal G \cap B_{\rho_0/2}(x_0))$
	\begin{equation} \label{L1}
	\mathcal L_j[w]:=\mathcal A_j[w] - M \frac{\langle \nabla z_j,\nabla w \rangle}{z_j}
	+2a_{0,j,w}+\left(m+2-p\right) a_{1,j,w},
	\end{equation} 
	where $M=M(m,N,p,q,\sigma,\tau)>0$ denotes a large constant (see \eqref{mdeff} in Lemma~\ref{lm11}).

	Let $\eta \in C^{\infty}_c(\mathbb{R}^N)$ be such that $0 \leq \eta \leq 1$, $\mathrm{Supp} (\eta) \subset B_{1/2}(0)$ and $\eta \equiv 1$ in $B_{1/3}(0)$. 
	We define 
	\begin{equation} \label{pf1}
	\phi(x)=\eta(\rho_0^{-1} (x-x_0))\quad \text{for all }x\in \mathbb R^N. 
	\end{equation}
	Let $\omega$ denote the following open set
	\begin{equation} \label{omegadef} \omega:=\mathcal G\cap \{x\in  B_{\rho_0/2}(x_0):\ (x-x_0)/\rho_0\in \mathrm{Int}\,(\mathrm{Supp} (\eta))\}.\end{equation}
	Note that there exists a positive constant $c'=c'(N)$ such that 
	\begin{equation} \label{c'} |D^2 \phi| \leq c' \rho_0^{-2}\ \ \text{and }\ \ |\nabla \phi| \leq c' \rho_0^{-1} \phi^{\frac{1}{2}}\quad
	\text{for every } x\in \omega.\end{equation} 
	The aim of Lemma~\ref{lm12} in Section~\ref{app1} is to compute
	$ \mathcal L_j[w_j]$ for $w_j$ given by  
	\begin{equation} \label{wj12} 
	w_j:=\phi^{2 \alpha_j} z_j \ \text{ with }\ \alpha_1=1/(2k)\ \ \text{and }\ \alpha_2=1/2.\end{equation}  
	Then, in Lemma~\ref{lm11}, we obtain an upper bound estimate of $ \mathcal L_j[w_j]$, proving that there exist positive constants $d_i=d_i(m,N,p,q,\sigma, \tau)$ with $i=0,1,2,3$ such that for $j=1,2$, we have
	\begin{equation} \label{up}
	\mathcal{L}_j[w_j]\leq \left[-d_0(j-1) z_j^2+d_1 \frac{z_j}{\phi |x|^2}-d_2 \left(a_{1,j}(x)\right)^2 -d_3 \frac{|\nabla z_j|^2}{z_j} \right] \phi^{2 \alpha_j}\ 
	\text{for all }x\in \omega.
	\end{equation}
	Since $\alpha_2=1/2$ and $w_2=\phi z_2$, for $j=2$, we obtain that  
	\begin{equation} \label{max} \mathcal{L}_2[w_2]\leq z_2\left(-d_0\,w_2+d_1|x|^{-2}\right)\quad \mathrm{for\ all }\ x\in \omega.\end{equation}
	Using that $w_2|_{\partial\omega}=0$, there exists 
	$x^*\in \omega$ such that $\max_{x\in \overline{\omega}}w_2(x)=w_2(x^*)>0$. Then, since
	$\nabla w_2(x^*)=0$ and $(D^2 w_2)(x^*)$ is negative semi-definite, we find that 
	\begin{equation} \label{max1} \mathcal{L}_2[w_2](x^*)= \mathcal{A}_2[w_2](x^*)=-(\Delta w_2)(x^*)-(p-2)
	\frac{\langle (D^2 w_2)(\nabla u_2), \nabla u_2 \rangle}{|\nabla u_2|^2}(x^*)\geq 0. 
	\end{equation} We show that $\mathcal{A}_2[w_2](x^*)\geq 0$. 
	Indeed, let $\lambda_1,\ldots, \lambda_N$ denote the eigenvalues of 
	$ (D^2 w_2)(x^*)$, the Hessian of $w_2$ at $x^*$. 
	Since $ (D^2 w_2)(x^*)$ is negative semi-definite, we have $\lambda_j\leq 0$ for every $j=1,\ldots,N$. 
	We assume the eigenvalues are arranged such that 
	$\lambda_1\leq \lambda_2\leq \ldots \leq \lambda_N\leq 0$. Now, the
	Rayleigh--Ritz Theorem applied to the real symmetric matrix $ (D^2 w_2)(x^*)$ yields that
	\begin{equation} \label{ritz}\lambda_1|\boldsymbol{\xi}|^2=\left(\min_{1\leq j\leq N} \lambda_j\right) | \boldsymbol{\xi}|^2 
	\leq \langle (D^2 w_2)(x^*)\,\boldsymbol{\xi},\boldsymbol{\xi} \rangle \leq \left(\max_{1\leq j\leq N} \lambda_j\right) | \boldsymbol{\xi}|^2 
= \lambda_N |\boldsymbol{\xi}|^2\end{equation} 	
for every $\boldsymbol{\xi}\in \mathbb R^N$. 
Since $(\Delta w_2)(x^*)=\sum_{i=1}^N \lambda_i$, using \eqref{ritz} with $\boldsymbol{\xi}=(\nabla u_2) (x^*)$, we obtain that  
	$$-(\Delta w_2)(x^*)-(p-2)
	\frac{\langle (D^2 w_2)(\nabla u_2), \nabla u_2 \rangle}{|\nabla u_2|^2}(x^*) \geq \left\{\begin{aligned}
	& -(p-1)\lambda_1-\sum_{i=2}^{N} \lambda_i && \mathrm{if }\ 1<p\leq 2,&\\
	& -\sum_{i=1}^{N-1}\lambda_i-(p-1)\lambda_N && \mathrm{if }\ 2<p< \infty.&
	\end{aligned} \right.
	$$ This proves the inequality in \eqref{max1}.  
	Letting $x=x^*$ in \eqref{max} and using \eqref{max1}, we arrive at 
	\begin{equation} \label{max2} w_2(x^*)\leq \left(d_1/d_0\right) |x^*|^{-2},
	\end{equation} where $d_0$ and $d_1$ are positive constants depending only on $m,N,p,q,\sigma$ and $\tau$. Recall that 
	$w_2(x^*)=\max_{x\in \overline{\omega}}w_2(x)$ and $|x^*|\geq |x_0|/2$. Since $\eta\equiv 1$ on $B_{1/3}(0)$, we have 
	$\phi(x_0)=1$. Hence, from \eqref{defu2} and \eqref{max2}, 
	we obtain that 
	$$ \frac{|\nabla u_1(x_0)|^2}{(u_1(x_0))^2}=|\nabla u_2(x_0)|^2=\phi(x_0)|\nabla u_2(x_0)|^2\leq w_2(x^*)\leq \frac{4 d_1}
	{d_0|x_0|^{2}}.
	$$ This proves the assertion of \eqref{gu1a} for $x=x_0$ arbitrary in $ \Omega\setminus\{0\}$ with 
	$|\nabla u_1(x_0)|\not=0$. The proof of Lemma~\ref{lm1} is thus complete.   
\end{proof}

\vspace{0.2cm}
{\em Proof of Theorem~{\rm\ref{thm1}} completed.} By taking $j=1$ and $\alpha_1=1/(2k)$ in \eqref{up}, we find that 
\begin{equation} \label{j22}
\mathcal L_1[w_1]\leq d_1 \phi^{\frac{1}{k}-1}|x|^{-2}z_1-d_2 u_1^{2q} \phi^{\frac{1}{k}}|x|^{-2(\tau+\sigma)} z_1^{m+2-p}
\quad \mathrm{for\ all }\ x\in \omega. 
\end{equation}
In \eqref{j22} we use Lemma~\ref{lm1} and $w_1=\phi^{1/k}z_1$ to conclude that 
\begin{equation} \label{j23}
\mathcal L_1[w_1]\leq  \phi^{\frac{1}{k}-1}|x|^{-2}z_1\left(d_1-C_1^{-2q}d_2 |x|^{2\ell} w_1^k\right)
\quad \mathrm{for\ all }\ x\in \omega, 
\end{equation} where $C_1>0$ is the constant appearing in \eqref{gu1a}, while $k$ and $\ell$ are given by \eqref{kl}. 

Let $x^*\in \omega$ be such that $\max_{x\in \overline{\omega}}w_1(x)=w_1(x^*)>0$. As before, we arrive at 
\begin{equation} \label{max11} \mathcal{L}_1[w_1](x^*)= \mathcal{A}_1[w_1](x^*)=\left[-\Delta w_1-(p-2)
\frac{\langle (D^2 w_1)(\nabla u_1), \nabla u_1 \rangle}{|\nabla u_1|^2}\right](x^*)\geq 0. 
\end{equation} 
Letting $x=x^*$ in \eqref{j23} and using \eqref{max11}, we arrive at 
\begin{equation} \label{max22} w_1(x^*)\leq \left(\frac{d_1}{d_2}\right)^{\frac{1}{k}} C_1^{\frac{2q}{k}}|x^*|^{-\frac{2\ell}{k}},
\end{equation} where $C_1$, $d_1$ and $d_2$ are positive constants depending only on $m,N,p,q,\sigma$ and $\tau$. 
Recall that 
$w_1(x^*)=\max_{x\in \overline{\omega}}\phi (x)|\nabla u_1(x)|^2$ and 
$|x^*|\geq |x_0|/2$. Since
$\phi(x_0)=1$, from \eqref{max22}, 
we obtain that 
$$ |\nabla u_1(x_0)|^2=\left(\phi(x_0)\right)^{1/k}|\nabla u_1(x_0)|^2\leq w_1(x^*)\leq C^2
|x_0|^{-\frac{2\ell}{k}},
$$ where $C:=\left(2^\ell C_1^q \sqrt{d_1/d_2}\right)^{1/k}$ is a positive constant depending only on $m,N,p,q,\sigma$ and $\tau$. 
This proves the assertion of \eqref{gu2a} for $x=x_0$ arbitrary in $ \Omega\setminus\{0\}$ with 
$|\nabla u_1(x_0)|\not=0$. The proof of our Theorem~\ref{thm1} is now finished.   
$\hfill \square$

\begin{lemma} \label{gu3}
	Let \eqref{GG} hold  and $\Omega_1  $ be any domain in $\mathbb{R}^N$.  
	There exists a positive constant $C_1=C_1(m,N,p,q,\sigma, \tau)$ such that any positive solution $u$ of \eqref{le} in $\Omega_1$ satisfies
	\begin{equation} \label{gue3a}
	|\nabla u (x)| \leq \frac{C_1  u(x)}{{\rm dist} (x, \partial \Omega_1)} \ \ \mbox{and } \ \ |\nabla u (x)| \leq C_1 {\rm dist}(x, \partial \Omega_1)^{-\frac{\ell}{k}} \quad \mbox{for all } x \in \Omega_1.
	\end{equation}
\end{lemma}

\begin{proof} The claim follows by taking $\rho_0={\rm dist}(x_0, \partial \Omega_1)$ rather than $\rho_0=|x_0|$ in the proofs of Theorem~\ref{thm1} and Lemma~\ref{lm1}. 
\end{proof}

The Liouville-type results in \cite{FarSer}*{Theorem 2,3} are improved by the following.

\begin{corollary}[Liouville-type theorem] \label{guec1}
	Let  \eqref{GG} hold. Any $C^1(\mathbb{R}^N)$ positive solution of \eqref{le} in $\mathbb{R}^N$ must be identically constant.
\end{corollary}

\begin{proof}
	We follow \cite{Lions}*{Corollary IV.2}. Fix $x_1 \in \mathbb{R}^N$. Let $R>0$ be arbitrary and $\Omega_1=B_{R}(x_1)$ in Lemma~\ref{gu3}. Then for $C_1$ as in Lemma~\ref{gu3}, we find that
	$|\nabla u (x_1)| \leq C_1 R^{-\frac{\ell}{k}}$.
	Letting $R \rightarrow \infty$, we obtain that $|\nabla u(x_1)|=0$.  Since $x_1$ was arbitrary, we conclude the claim.
\end{proof}

\begin{corollary}[{\em A priori} estimates of solutions] \label{ch4ap1}
	Let  \eqref{GG} hold and $u$ be an arbitrary positive solution of \eqref{le}. If $k>\ell$, then 
	$u\in L^\infty_{\rm loc}(\Omega)$. If 
	$k\leq \ell$, then there exists a positive constant $c_1$ depending only on $m,N,p,q,\sigma$ and $\tau$ such that for all $ 0<|x|<r_0$ with $\overline{B_{2r_0}(0)}\subset \Omega$, it holds
	\begin{equation} \label{ooo} u(x)\leq \left\{ \begin{aligned}
	& \max_{\partial B_{r_0}(0)} u+ c_1 \log (r_0/|x|)  && \text{if } k=\ell; &\\
	& \max_{\partial B_{r_0}(0)} u+ c_1 |x|^{1-\frac{\ell}{k}} && \text{if } k<\ell. &
	\end{aligned} \right. \end{equation} 
	\end{corollary}

\begin{proof} The proof follows similarly to that in \cite{BVGH}*{Section 2.2}.
	Fix $x \in B_{r_0}(0) \setminus \{ 0 \}$ and let $X=r_0x/|x|$.  Using Theorem~\ref{thm1}, we find that
	$$ 
	|u(x)-u(X)| \leq  |x-X| \int^1_0 | (\nabla u)(tx+(1-t)X) |\, dt 
	 \leq C |x-X| \int^1_0 \left( t|x| + (1-t)r_0 \right)^{- \frac{\ell}{k}} \, dt.
	$$
The conclusion follows by integration. 
\end{proof}

For a different proof of the second inequality in \eqref{ooo} in the case $\sigma=\tau=0<m<2=p$,  we refer to 
Ching and C\^{\i}rstea \cite{CC}*{Lemma~3.4}, where
a comparison with a suitable boundary blow-up super-solution is used.  
Unlike the case $m=0$, it has been observed in \cite{CC}*{Remark~3.5} that the term  $\max_{\partial B_{r_0}(0)} u$ arising in the estimate \eqref{ooo}
is due to the introduction of the gradient factor $|\nabla u|^m$ in \eqref{le} and cannot be removed. 

\section{Auxiliary results} \label{app1}

In Lemma~\ref{smp11} we prove that the strong maximum principle is applicable for the non-negative solutions of \eqref{le} when \eqref{GG} holds. The proof of Lemma~\ref{lm1} was essentially based on the estimate 
of \eqref{up}, which follows from Lemma~\ref{lm11} and Remark~\ref{remark2}. We present here the proof of Lemma~\ref{lm11}, which is quite intricate and relies on Lemmas~\ref{lm12} and \ref{lm13}.

\begin{lemma}[Strong maximum principle] \label{smp11} Assume that  \eqref{GG} holds. 
	For any non-negative solution $u$ of \eqref{le}, either $u>0$ in $\Omega^*$ or $u\equiv 0$ in $\Omega^*$.  
\end{lemma}

\begin{proof}
	Let $\varepsilon_0>0$ be small so that $B_{\varepsilon_0}(0)\subset \Omega$. For every $x\in \Omega\setminus\{0\}$, we can find 
	$R>\varepsilon$ and $\varepsilon\in (0,\varepsilon_0)$ such that $x\in \Omega_{R,\varepsilon}$, where $\Omega_{R,\varepsilon}:=(\Omega\cap B_R(0))\setminus \overline{B_\varepsilon(0)}$. Hence, the claim follows by checking 
	that the strong maximum principle holds in $\Omega_{R,\varepsilon}$ for every $R>\varepsilon$ and any $\varepsilon\in (0,\varepsilon_0)$.   
	We use \cite{Pucci}*{Theorem~5.4.1} for (5.4.1) on $\Omega_{R,\varepsilon}$, namely
	$$
	\partial_{x_j} \left\{ a_{ij}(x,u) A(|\nabla u|) \partial_{x_j}u\right\}+B(x,u,\nabla u)\leq 0
	$$
	with $a_{ij}(x,u)=|x|^\sigma\delta_{ij}$, where $\delta_{ij}$ denotes the Kronecker delta, $A(t)=t^{p-2}$ for any $ t\geq 0$ and 
	$$B(x,z,\boldsymbol{\xi})=-|x|^{-\tau} z^q|\boldsymbol{\xi}|^m \quad \text{for } x\in \Omega_{R,\varepsilon}, \ z\geq 0\ \text{ and } \boldsymbol{\xi}\in \mathbb R^N.$$ We have  $A\in C^1(\mathbb R^+)$ and $\lim_{t\searrow 0} tA'(t)/A(t)=p-2>-1$ so that (A1)' and (5.4.3) in \cite{Pucci} hold. 
	It is easy to check (A2) in \cite{Pucci}*{p.~3}. We also have (5.4.4) using \cite{Pucci}*{Remark 3, p.~117}.  It remains to check (B1) and (F2) in \cite{Pucci}*{p. 107}.  Since $k>0$ from \eqref{GG}, we can find $s$ such that 
	$ s>\max\{(p-1)/q,1\} $ and $ms'>p-1$, where $s'$ denotes the H\"older conjugate of $s$, that is $s':=s/(s-1)$. 
	By Young's inequality, we have 
	$$
	z^q|\boldsymbol{\xi}|^m \leq \frac{z^{qs}}{s}+\frac{|\boldsymbol{\xi}|^{ms'}}{s'} \leq \frac{z^{qs}}{s}+\frac{|\boldsymbol{\xi}|^{p-1}}{s'}\quad
	\text{for all } z\in \mathbb R^+\ \text{and } \boldsymbol{\xi}\in \mathbb R^N\ \text{with } |\boldsymbol{\xi}|\leq 1. 
	$$ 
	Hence, Condition (B1) holds with $\Phi(|\boldsymbol{\xi}|)=
	|\boldsymbol{\xi}|A(|\boldsymbol{\xi}|)=|\boldsymbol{\xi}|^{p-1}$, $\kappa=\max(R^{-\tau},\varepsilon^{-\tau})/s'$ and $f(z)=(\max(R^{-\tau},\varepsilon^{-\tau})/s) \, z^{qs}$ satisfying 
	(F2).  We can now apply Theorem~5.4.1 in \cite{Pucci} to conclude the proof of Lemma~\ref{smp11}.
	\end{proof}

In the rest of this section, we work in the framework and notation of Lemma~\ref{lm1}. 
Our main aim is to prove Lemma~\ref{lm11}, which gives an estimate from above for $\mathcal{L}_j[w_j]$, where
the operator $\mathcal{L}_j[w]$ and $w_j$ are defined in \eqref{L1} and \eqref{wj12}, respectively. An important ingredient is Lemma~\ref{lm12} in which we 
evaluate $\mathcal{L}_j[w_j]$, see \eqref{ch4aa}. For this purpose, we introduce the following. 

\vspace{0.2cm}
{\bf Notation.} 
We define $\mathcal Q_j(x)$, $\Theta_j(x)$ and $\mathcal{X}_j(x)$ for $j=1,2$ and $x\in \omega$ as follows:
\begin{equation} \label{qxe} \left\{ \begin{aligned}
&{\mathcal Q}_j(x):=(\sigma+\tau) \frac{\langle x,\nabla u_j \rangle}{|x|^2}+\alpha_j \left(m+2-p\right)\frac{\langle \nabla u_j,\nabla \phi\rangle}{\phi},\\
& \Theta_j(x):=\Psi_j(x)-4\alpha_j \frac{ \langle \nabla \phi, \nabla z_j \rangle }{\phi},\ 
\text{where } \Psi_j:= 
-\frac{2\alpha_j\,z_j}{\phi}\left( (2 \alpha_j - 1)\frac{|\nabla \phi|^2}{\phi}+ \Delta \phi \right),\\
& {\mathcal X}_j(x):=\Theta_j- M \frac{\langle \nabla z_j, \nabla w_j \rangle}{z_j \phi^{2\alpha_j}}+ \frac{a_{2,w_j}+a_{3,j,w_j}}{\phi^{2 \alpha_j}}
- 2 \,\langle \nabla \left(a_{2,j}+ a_{3,j}\right), \nabla u_j\rangle.
\end{aligned} \right.
\end{equation}

\begin{lemma} \label{lm12}
	Let  \eqref{GG} hold.  Then, for every $x\in \omega$ and $j=1,2$, the following holds:
	\begin{equation} \label{ch4aa}
	\mathcal{L}_j[w_j]= \left[ \mathcal X_j -2|D^2 u_j|^2+ 4 \alpha_j a_{0,j}   \frac{\langle \nabla u_j,\nabla \phi\rangle}{\phi}
	+ 2\left(\mathcal Q_j-\frac{h_j'(u_j)}{h_j(u_j)} z_j\right) a_{1,j}
	\right] \phi^{2 \alpha_j},
	\end{equation} 
	where $a_{0,j}$ and $a_{1,j}$ are given in \eqref{a01}. 
\end{lemma}

\begin{proof} By the definition of $\mathcal{L}_j[w]$ in \eqref{L1}, we have
	\begin{equation} \label{LL1}
	\mathcal{L}_j[w_j]=-\Delta w_j - M \frac{\langle \nabla z_j, \nabla w_j \rangle}{z_j}+2a_{0,j,w_j}+(m+2-p)a_{1,j,w_j}+a_{2,w_j}+a_{3,j,w_j}.
	\end{equation}
	Since $ \Delta z_j = 2|D^2 u_j|^2 + 2\langle \nabla (\Delta u_j), \nabla u_j \rangle$, by using \eqref{vee}, we find that
	\begin{equation}\label{eo1}
	\Delta z_j = 2 |D^2 u_j|^2  + 2\sum^3_{i=0} \langle \nabla a_{i,j}, \nabla u_j\rangle \quad \mbox{for all } x \in \omega.
	\end{equation}
	Recall that $w_j =\phi^{2\alpha_j} z_j$. Hence, by the product rule and \eqref{eo1}, for all $x \in \omega$, we obtain that 
	\begin{equation} \label{eo}
	-\Delta w_j =  \left(\Theta_j -\Delta z_j\right)\phi^{2 \alpha_j} =\left(\Theta_j -2|D^2 u_j|^2-2 \sum^3_{i=0} \langle \nabla a_{i,j}, \nabla u_j\rangle  \right)\phi^{2 \alpha_j}. 
	\end{equation}
	Using \eqref{eo} in \eqref{LL1}, we get that
	\begin{equation} \label{mle1}
	\mathcal{L}_j[w_j] = \left(\mathcal X_j -2 |D^2 u_j|^2+ \mathcal  E_{0,j}+ \mathcal  E_{1,j}
	\right) \phi^{2 \alpha_j} \quad \text{in } \omega,
	\end{equation}
	where $ \mathcal  E_{0,j}$ and $\mathcal  E_{1,j}$ are defined in $\omega$ as follows
	\begin{equation} \label{eij} 
	\mathcal E_{0,j}:= 2\frac{ a_{0,j,w_j}}{\phi^{2 \alpha_j}}- 2 \langle \nabla a_{0,j}, \nabla u_j\rangle\ \text{and }  \mathcal E_{1,j}:= (m+2-p)\,\frac{ a_{1,j,w_j}}{\phi^{2 \alpha_j}}- 2 \langle \nabla a_{1,j}, \nabla u_j\rangle.
	\end{equation}
	We now evaluate the terms $\mathcal E_{0,j}$ and $\mathcal E_{1,j}$ for $j=1,2$. We use that 
	\begin{equation}\label{abi0}
	\langle \nabla u_j, \nabla w_j \rangle =\left( 2 \alpha_j  z_j \frac{\langle \nabla u_j, \nabla \phi \rangle}{\phi} + \langle \nabla z_j, \nabla u_j\rangle\right) \phi^{2 \alpha_j} \quad \mbox{in } \omega.
	\end{equation}
	Hence, for every $x\in \omega$, we have
	\begin{equation} \label{b0j}
	\mathcal E_{0,j}(x)= 4 \alpha_j a_{0,j}  \frac{\langle \nabla u_j, \nabla \phi \rangle}{\phi} \quad \text{and}\quad
	\mathcal E_{1,j} (x) =
	2\left(\mathcal Q_j-\frac{h_j'(u_j)}{h_j(u_j)} z_j\right) a_{1,j} .
	\end{equation}
	Using \eqref{b0j} into \eqref{mle1}, we reach \eqref{ch4aa}. This ends the proof of Lemma~\ref{lm12}.  \end{proof}

\begin{lemma} \label{lm13}
	Let  \eqref{GG} hold. For every $M>3|p-2|/2$, there exist positive constants $\beta_{1,j}$ and $\beta_{2,j}$, 
	depending on $m,N,p,q,\sigma$ and $M$ such that \begin{equation} \label{josh}
	\mathcal X_j(x)\leq \beta_{1,j} \frac{z_j}{|x|^2 \phi}-\beta_{2,j} \,\frac{|\nabla z_j|^2}{z_j}\quad \text{for all } x\in \omega \ \mbox{and } j=1,2.
	\end{equation}  
\end{lemma}

\begin{proof} Let $c'>0$ be as in  \eqref{c'}. For any $M>3|p-2|/2$, we fix $\varepsilon \in (0,\min_{j=1,2}\left(3\alpha_jc'\right)^{-1})$ such that $\beta_{2,j}>0$ for $j=1,2$, where we define
	\begin{equation} \label{beta2} \beta_{2,j}:=(1-3\alpha_j c'\varepsilon)M-6\alpha_jc'(1+|p-2|)\varepsilon-3|p-2|/2.
	\end{equation} 
	
	\vspace{0.1cm}
	{\bf Claim:} {\em There exists $\bar c_j>0$ depending only on $m,N,p,q,\sigma$ such that for all $x\in \omega$
		\begin{equation} \label{geee}  \mathcal X_j(x)\leq \bar{c}_j \frac{z_j}{|x|^2 \phi} -\left(M-\frac{3|p-2|}{2}\right)\frac{|\nabla z_j|^2}{z_j}+
		2\alpha_j(M+2+2|p-2|) \frac{|\nabla z_j| |\nabla \phi|}{\phi}.
		\end{equation}}
 
	Assume that the Claim has been proved. 
	From \eqref{c'} and Young's inequality with $\varepsilon$, we get
	\begin{equation} \label{zj1}
	\frac{ |\nabla z_j||\nabla \phi|}{\phi} \leq \frac{3c'}{2}  \frac{ |\nabla z_j| }{|x|\phi^{1/2}} 
	\leq \frac{3c'}{2} \left( \varepsilon  \frac{|\nabla z_j|^2}{z_j} + \frac{1}{4 \varepsilon} \frac{z_j}{|x|^2\phi}\right)\quad \mbox{for every } x \in \omega.
	\end{equation}
	Using \eqref{zj1} into \eqref{geee}, we reach \eqref{josh} with $\beta_{2,j}$ given by \eqref{beta2} and $\beta_{1,j}$ defined by 
	\begin{equation} \label{beta1}
	\beta_{1,j}:=\overline{c}_j+3\alpha_jc'(M+2+2|p-2|)/(4\varepsilon).
	\end{equation}
	
	{\em Proof of Claim.} If we define $ \mathcal  E_{2,j}$ and $\mathcal  E_{3,j}$ by 
	$$  \mathcal E_{2,j}:= \frac{a_{2,w_j}}{\phi^{2 \alpha_j}}- 2 \langle \nabla a_{2,j}, \nabla u_j\rangle\quad
	\text{ and }\quad  \mathcal E_{3,j}:= \frac{a_{3,j,w_j}}{\phi^{2 \alpha_j}}- 2 \langle \nabla a_{3,j}, \nabla u_j\rangle, 
	$$
	then the definition of $\mathcal X_j$ in \eqref{qxe} yields that 
	\begin{equation} \label{xdef}
	\mathcal X_j=\Theta_j - M \frac{\langle \nabla z_j, \nabla w_j \rangle}{z_j \phi^{2\alpha_j}}+\mathcal E_{2,j}+\mathcal E_{3,j}. 
	\end{equation}
	
	From  \eqref{xdef}, we will derive \eqref{geee} by bounding from above $\Theta_j - (M \langle \nabla z_j, \nabla w_j \rangle)/(z_j \phi^{2\alpha_j})$, as well
	as $\mathcal E_{2,j}$ and $\mathcal E_{3,j}$ 
	in \eqref{fir}, \eqref{c''} and \eqref{e3j}, respectively.
	
	\vspace{0.1cm}
	By the definition of $\Theta_j$ in \eqref{qxe} and 
	$(1/z_j) \phi^{-2\alpha_j}  \nabla w_j= (1/z_j)\nabla z_j+(2\alpha_j/\phi) \nabla \phi$, we find that
	\begin{equation} \label{b9j}
	\Theta_j-M\frac{\langle \nabla z_j, \nabla w_j \rangle}{z_j \phi^{2\alpha_j}} = 
	\Psi_j
	-M  \frac{|\nabla z_j|^2}{z_j}-2 \alpha_j (M+2) \frac{ \langle \nabla z_j, \nabla \phi\rangle}{\phi} \quad \text{in }\omega.
	\end{equation}
	From \eqref{c'} and $ |\Delta \phi|\leq \sqrt{N} |D^2 \phi| $, we have $|\Delta \phi| \leq 9\sqrt{N}c' /(4|x|^2)$ in $\omega$. 
	Using
	\eqref{b9j} and denoting $\widehat c_j:=(9\alpha_j c' /2)\left(\sqrt{N}+c'|2\alpha_j-1|\right)$, we arrive at 
	\begin{equation} \label{fir}
	\Theta_j-M\frac{\langle \nabla z_j, \nabla w_j \rangle}{z_j \phi^{2\alpha_j}}\leq 
	\widehat c_j\frac{z_j}{|x|^2\phi}
	-M  \frac{|\nabla z_j|^2}{z_j}+2 \alpha_j (M+2) \frac{| \nabla z_j| |\nabla \phi |}{\phi}.  
	\end{equation}
	
	Next, we will bound $ \mathcal  E_{2,j}$ and $\mathcal  E_{3,j}$ from above. 
	Using the identity $$ \left\langle \nabla \left( \frac{\langle x, \nabla u_j\rangle}{|x|^2}\right), \nabla u_j\right\rangle = \frac{z_j}{|x|^2} + \frac{\langle x, \nabla z_j\rangle}{2 |x|^2} - \frac{2 \langle x, \nabla u_j \rangle^2}{|x|^4},$$ we obtain that
	\begin{equation} \label{b2j}
	\mathcal E_{2,j}  (x)
	=2 \sigma\left[ \left(1 - 2 \frac{\langle x, \nabla u_j\rangle^2}{|x|^2 z_j} \right)\phi-  \alpha_j \langle x, \nabla \phi\rangle \right] \frac{ z_j}{|x|^2\phi}.
	\end{equation}
	Therefore, using the constant $c'$ in \eqref{c'}, we find that 
	\begin{equation} \label{c''}
	\mathcal E_{2,j}  (x)\leq 3 |\sigma| \left(2+c'\alpha_j\right)\frac{z_j}{|x|^2\phi}\quad \text{for every } x\in \omega.
	\end{equation}
	Now, using the identity that $$  D^2 (\xi \zeta ) \equiv \xi D^2 \zeta + \zeta D^2 \xi + (\nabla \zeta)(\nabla \xi)^\intercal + (\nabla \xi)(\nabla \zeta)^\intercal $$  for any two $\xi,\zeta$ twice continuously differentiable functions, we find that
	$$
	D^2 w_j = \phi^{2 \alpha_j} D^2 z_j + z_j D^2(\phi^{2 \alpha_j}) + (\nabla \phi^{2 \alpha_j}) (\nabla z_j)^{\intercal} + (\nabla z_j)(\nabla \phi^{2 \alpha_j})^{\intercal} \quad \mbox{in } \omega.
	$$
	Since the chain rule implies that $$ D^2 (\phi^{2 \alpha_j}) = 2 \alpha_j (2 \alpha_j - 1) \phi^{2 \alpha_j - 2} (\nabla \phi)(\nabla \phi)^{\intercal} + 2 \alpha_j \phi^{2 \alpha_j - 1} D^2 \phi,$$ we arrive at
	\begin{equation} \label{b3jaa}
	\frac{a_{3,j,w_j}}{\phi^{2\alpha_j}} =a_{3,j,z_j}-4\alpha_j(p-2) \frac{ \langle \nabla u_j, \nabla \phi \rangle
		\langle \nabla u_j, \nabla z_j \rangle}{z_j \phi}+\Upsilon_j,
	\end{equation}
	where we define $\Upsilon_j$ by
	$$ \Upsilon_j:=
	-\frac{2\alpha_j(p-2) }{\phi}\left[ \langle(D^2 \phi)(\nabla u_j), \nabla u_j\rangle 
	+  (2 \alpha_j - 1) \frac{\langle \nabla \phi, \nabla u_j \rangle^2}{\phi} \right] .
	$$
	In view of \eqref{c'}, by taking $c''_j=9\alpha_j |p-2|c' (1+c'|2\alpha_j-1|)/2$, we have
	\begin{equation} \label{ups} \Upsilon_j(x)\leq c''_j\frac{z_j}{|x|^2 \phi}\quad \text{for all } x\in \omega.  
	\end{equation}
	Using the identity that  $$ \left\langle \nabla \left( \frac{\langle \nabla z_j, \nabla u_j \rangle}{z_j}\right), \nabla u_j \right\rangle = 
	- \frac{\langle \nabla z_j, \nabla u_j\rangle^2}{z_j^2} + \frac{|\nabla z_j|^2}{2z_j} + \frac{\langle (D^2 z_j)(\nabla u_j), \nabla u_j \rangle}{z_j},$$
	for all $x \in \omega$, we find that
	\begin{equation} \label{b3j}
	-2\langle \nabla a_{3,j},\nabla u_j\rangle 
	=-a_{3,j,z_j}+(p-2)\left( \frac{1}{2}-\frac{\langle \nabla z_j, \nabla u_j \rangle^2}{z_j |\nabla z_j|^2} \right)\frac{|\nabla z_j|^2}{z_j}.
	\end{equation}
	By adding \eqref{b3jaa} and \eqref{b3j}, then using \eqref{ups}, we infer that 
	\begin{equation} \label{e3j}
	\mathcal E_{3,j}(x)\leq \frac{3|p-2|}{2} \frac{|\nabla z_j|^2}{z_j}+c''_j\frac{z_j}{|x|^2 \phi} +4\alpha_j |p-2| \frac{ |\nabla \phi | |\nabla z_j|}{\phi}
	\quad \text{for all } x\in \omega.
	\end{equation}
	If $\overline{c}_j:=\widehat c_j+3|\sigma|(2+c'\alpha_j)+c_j''$, then by adding \eqref{fir}, \eqref{c''} and \eqref{e3j}, we conclude \eqref{geee}. 
	This proves the Claim, thus completing the proof of Lemma~\ref{lm13}.  
\end{proof}

\begin{lemma} \label{lm11}
	Let  \eqref{GG} hold. Fix $\theta \in \mathbb{R}$ such that $0<\theta<\min \left(k/(2(p-1)),1/N \right)$. There exists a positive constant $M=M(m,N,p,q,\sigma, \tau)$ such that for $i=0,1,2,3$  we can find positive constants $d_i = d_i(m,N,p,q,\sigma, \tau)$ so that for $j=1,2$, we have
	\begin{equation} \label{up2}
	\mathcal{L}_j[w_j]\leq \left[-d_0(j-1)  z_j^2+d_1 \frac{z_j}{\phi |x|^2}-d_2  a_{1,j}^2(x)-d_3 \frac{|\nabla z_j|^2}{z_j} - 2\mathcal{W}_j(x) \right] \phi^{2 \alpha_j},
	\end{equation}
	for every $x\in \omega$ and $\mathcal{W}_j$ is given by
	\begin{equation}
	\mathcal{W}_j(x):= h_j' (u_j) z_j f(x, z_j) + 2 \theta a_{0,j}(x) \,a_{1,j}(x)  \quad \mbox{for every } x \in \omega.
	\end{equation}
\end{lemma}

\begin{proof}
	Recall that $\alpha_1=1/(2k)$, $\alpha_2=1/2$, and $\beta_{1,j},\beta_{2,j}$ for $j=1,2$ are given respectively by \eqref{beta1} and \eqref{beta2}.   We fix $\varepsilon>0$ small (see \eqref{edeff}) and $M> 3 |p-2|/2$ large (see \eqref{mdeff}) both depending only on $m,N,p,q,\sigma,\tau$.  From Lemmas~\ref{lm12} and \ref{lm13}, using \eqref{josh} in \eqref{ch4aa}, we find that
	\begin{equation} \label{lm98}
	\frac{\mathcal{L}_j[w_j]}{\phi^{2 \alpha_j}}\leq \beta_{1,j} \frac{z_j}{|x|^2 \phi}-\beta_{2,j} \,\frac{|\nabla z_j|^2}{z_j} - 2 \left( |D^2 u_j|^2 - 2 \theta a_{0,j} a_{1,j}\right) + \mathcal Z_j - 2 \mathcal W_j \quad \mbox{in } \omega,
	\end{equation} 
	where, we define $\mathcal Z_j(x)$ for all $x \in \omega$ and $j=1,2$ by
	\begin{equation} \label{zj56}
	\mathcal Z_j (x) :=\ 4 \alpha_j a_{0,j}   \frac{\langle \nabla u_j,\nabla \phi\rangle}{\phi}+ 2\mathcal Q_j (x) a_{1,j}.
	\end{equation}
	
	We will estimate $\left( |D^2 u_j|^2 - 2 \theta a_{0,j} a_{1,j}\right)$ from below in Claim 1 and $\mathcal Z_j $ from above in Claim~2.  These estimates will be used in \eqref{lm98} to obtain \eqref{up8}, from which \eqref{up2} follows immediately.
	
	\vspace{3 mm}
	{\bf Claim 1: } \emph{For all $x \in \omega$ and $j=1,2$, we have that}
	\begin{equation} \label{dd0}
	\begin{aligned}
	|D^2 u_j|^2-2 \theta a_{0,j} a_{1,j}\geq& \ \left[ (1-4 \varepsilon) (a_{0,j}^2+a_{1,j}^2) - \frac{\sigma^2 z_j}{\varepsilon|x|^2} - \frac{(p-2)^2}{4\varepsilon} \frac{|\nabla z_j|^2}{z_j}\right] \theta.
	\end{aligned}
	\end{equation}
	\vspace{3 mm}
	{\em Proof of Claim 1: }  
	Given arbitrary $b_i \in \mathbb{R}$ for $i \in \{ 0,1,2,3\}$, we have
	\begin{equation} \label{in3}
	\left( \sum^3_{i=0} b_i\right)^2 - 2 b_0 b_1 \geq (1-4 \varepsilon)(b_0^2+b_1^2) - \frac{1}{\varepsilon}b_2^2 - \frac{1}{\varepsilon}b_3^2 .
	\end{equation}
	Indeed, \eqref{in3} follows since for $(i,j) \in \{ 0,1\} \times \{ 2,3\}$, Young's inequality with $\varepsilon$ implies
	$$
	2 b_i b_j \geq -2 |b_i| |b_j| \geq - \left( 2 \varepsilon b_i^2 + \frac{1}{2 \varepsilon} b_j^2\right).
	$$
	We observe that 
	\begin{equation} \label{in4}
	a_{2,j}^2(x) \leq \frac{\sigma^2 z_j}{|x|^2} \quad \mbox{and } \  a_{3,j}^2(x) \leq \frac{(p-2)^2}{4} \frac{|\nabla z_j|^2}{z_j} \quad \mbox{for every } x \in \omega.
	\end{equation}
	Since $\theta \in (0,1/N)$, from the inequality $|D^2 u_j|^2 \geq (\Delta u_j)^2/N$ in $\omega$ for $j=1,2$ and \eqref{vee}, we find 
	\begin{equation} \label{in2}
	|D^2 u_j|^2 -2 \theta a_{0,j} a_{1,j} \geq \theta \left[\left( \sum_{i=0}^3 a_{i,j} \right)^2 -2  a_{0,j} a_{1,j} \right].
	\end{equation}
	Taking $b_i = a_{i,j}$ for $i=0,1,2,3$ in \eqref{in3}, then using \eqref{in4} and \eqref{in2}, we reach \eqref{dd0}.
	
	\vspace{3 mm}
	{\bf Claim 2: } \emph{ With $\mathcal Z_j$ defined in  \eqref{zj56}, we have 
		\begin{equation} \label{zbbb}
		\mathcal{Z}_j \leq \widetilde{c}_j \varepsilon a_{1,j}^2 + 6\alpha_j (j-1)(p-1) c' \varepsilon z_j^2 + \left( \frac{\widetilde{c}_j}{4 \varepsilon} + \frac{3
			\alpha_j (j-1)(p-1)c'}{2 \varepsilon} \right)\frac{z_j}{|x|^2 \phi}
		\end{equation}
		for all $x \in \omega$, where $\widetilde{c}_j :=2 |\sigma + \tau| + 3 \alpha_j |m+2-p| c'$.}
	
	\vspace{3 mm}
	{\em Proof of Claim 2: } From \eqref{omegadef} and \eqref{c'}, there exists a constant $c'=c'(N)>0$ such that 
	\begin{equation} \label{ccv} |\nabla \phi(x)| \leq \frac{3c'}{2|x|} (\phi(x))^{\frac{1}{2}}\quad \text{for all } x\in \omega.\end{equation} 
	Thus, we find that
	\begin{equation} \label{qja}
	|\mathcal{Q}_j (x)| \leq \frac{|\sigma + \tau| z_j^{\frac{1}{2}}}{|x|} + \frac{3\alpha_j |m+2-p|c' z_j^{\frac{1}{2}}}{2 |x| \phi^{\frac{1}{2}}} \leq  \frac{\widetilde{c}_j}{2} \frac{z_j^{\frac{1}{2}}}{|x| \phi^{\frac{1}{2}}} \quad \mbox{for every } x \in \omega.
	\end{equation}
	Using Young's inequality for $\chi_j=|a_{1,j}|$ or $\chi_j=z_j$ with $j=1,2$, we get
	\begin{equation} \label{zbim}
	\chi_j \frac{z_j^\frac{1}{2}}{|x| \phi^{\frac{1}{2}}} \leq \varepsilon \chi_j^2 + \frac{z_j}{4 \varepsilon |x|^2 \phi} \quad \mbox{for all } x \in \omega.
	\end{equation}
	Hence, using \eqref{ccv}, \eqref{qja} and \eqref{zbim}, we find for every $x \in \omega$
	\begin{equation} \label{ye9}
	\left\{
	\begin{aligned}
	&4 \alpha_j a_{0,j }\frac{\left| \langle \nabla u_j,\nabla \phi\rangle \right|}{\phi} \leq 3(j-1)(p-1) \alpha_j c'\left(2 \varepsilon   z_j^2 + \frac{1}{2 \varepsilon} \frac{z_j}{|x|^2\phi} \right), \\
	&2 | \mathcal{Q}_j (x)a_{1,j}(x)| \leq\widetilde{c}_j |a_{1,j}(x)|\frac{z_j^{\frac{1}{2}}}{\phi^{\frac{1}{2}}|x|} \leq \widetilde{c}_j \left( \varepsilon a_{1,j}^2(x) + \frac{z_j}{4 \varepsilon|x|^2 \phi}\right).
	\end{aligned}
	\right.
	\end{equation}
	Thus, we obtain \eqref{zbbb} from \eqref{ye9} as claimed.
	
	\vspace{3 mm}
	{\em Proof of Lemma~\ref{lm11} completed:}  
	By Claims 1 and 2, using \eqref{dd0} and \eqref{zbbb} in \eqref{lm98}, we get 
	\begin{equation} \label{up8}
	\mathcal{L}_j[w_j]\leq \left[-d_0(j-1)  z_j^2+d_{1,j} \frac{z_j}{\phi |x|^2}-d_{2,j}  a_{1,j}^2(x)-d_{3,j} \frac{|\nabla z_j|^2}{z_j} - 2\mathcal{W}_j(x) \right] \phi^{2 \alpha_j},
	\end{equation}
	where the constants $d_{1,j},d_{2,j},d_{3,j}$ and $d_0$ are given by
	\begin{equation} \label{dc1}
	\left\{
	\begin{aligned}
	d_{1,j}&= \beta_{1,j} + \frac{1}{\varepsilon} \left( 2 \theta \sigma^2 +\frac{\widetilde{c}_j}{4}+ \frac{3 \alpha_j (j-1)(p-1)c'}{2}\right), \quad d_{2,j}= 2 \theta - (8\theta+ \widetilde{c}_j)\varepsilon, \\
	d_{3,j} &= \beta_{2,j} - \frac{\theta(p-2)^2}{2 \varepsilon}, \quad d_0= (p-1)\left[ 2 \theta (p-1) - (3 c' + 8 \theta (p-1))\varepsilon \right].\\
	\end{aligned}
	\right.
	\end{equation}
	Now, we fix $\varepsilon>0$ depending only on $m,N,p,q,\sigma,\tau$ small enough such that 
	\begin{equation} \label{edeff}
	d_0>0, \quad 1-3 \alpha_j c' \varepsilon>0, \quad \mbox{and } d_{2,j}>0 \quad \mbox{for } j=1,2.
	\end{equation}
	Then we choose $M>3|p-2|/2$ depending only on $m,N,p,q,\sigma,\tau$ such that 
	\begin{equation} \label{mdeff}
	d_{3,j}>0, \quad \mbox{where $d_{3,j}$ is given by (\ref{dc1}) for $j=1,2$}.
	\end{equation}
	By taking $d_1=\displaystyle\max_{j=1,2} d_{1,j}$ and $d_i=\displaystyle\min_{j=1,2} d_{i,j}$ for $i=2,3$, we conclude \eqref{up2} from \eqref{up8}. 
\end{proof}

\begin{remark} \label{remark2} {\rm In Lemma~\ref{lm11}, we have $\mathcal{W}_j(x) \geq 0$ for every $x \in \omega$ and for $j=1,2$. Indeed, when $j=1$, we recall that $a_{0,j}(x)=0$ and $h_j'(u_j) \geq 0$ for every $x \in \omega$.  When $j=2$, we use that $h_2'(t)=-k h_2(t) \geq 0$ and our choice of $\theta$ ensures that
	\begin{equation} \label{ye8}
	h_2'(u_2)+2 \theta (p-1)h_2(u_2) = \left( 2 \theta (p-1)-k\right)h_2(u_2)  > 0.
	\end{equation}	}
\end{remark}

\section*{Acknowledgement}

The authors would like to thank the referee for a careful reading of the manuscript, useful suggestions and drawing our attention to a few 
typographical errors.

\end{document}